\documentclass[10pt,oneside,reqno]{amsart}
\usepackage{amsmath}
\usepackage{comment}
\usepackage{amsmath}
\usepackage{amssymb}
\usepackage{mathtools}

\numberwithin{equation}{section}
\theoremstyle{plain}
\newtheorem{Th}{Theorem}[section]
\newtheorem{Lemma}{Lemma}[section]
\newtheorem{Cor}{Corollary}[section]
\newtheorem{Prop}{Proposition}[section]

\theoremstyle{definition}
\newtheorem{Ques}{Question}

\newtheorem{Rem}{Remark}[section]
\newtheorem{?}{Problem}[section]

\newcommand{\Q}{\mathbb{Q}}
\newcommand{\bQ}{\overline{\mathbb{Q}}}
\newcommand{\Z}{\mathbb{Z}}
\newcommand{\N}{\mathbb{N}}

\newcommand{\Gal}[1]{\textrm{Gal}({#1})}
\newcommand{\Tr}{\textrm{Tr}}

\begin{document}
	\title{Sufficient conditions for a problem of Polya}
	\author{Abhishek Bharadwaj,  Veekesh Kumar, Aprameyo Pal,  and R. Thangadurai} 
	\address[Abhishek Bharadwaj]{Department of Mathematics, Jeffery Hall, 99 University Avenue, Queen's University, Kingston, Ontario, K7L 3N6, Canada}
	\address[Aprameyo Pal and R. Thangadurai]{Harish-Chandra Research Institute, HBNI, Chhatnag Road, Jhunsi, Prayagraj 211019, India.}
	\address[Veekesh Kumar]{Department of Mathematics, Indian Institute of Technology Dharwad, Chikkamalligawad village, Dharwad, Karnataka - 580007 }
 \email[Abhishek Bharadwaj]{atb4@queensu.ca}
	 \email[Aprameyo Pal]{aprameyopal@hri.res.in}
 \email[Veekesh Kumar]{veekeshk@iitdh.ac.in}
 \email[R. Thangadurai]{thanga@hri.res.in}
	
	%\maketitle
	\begin{abstract}
		Let $\alpha$ be a non-zero algebraic number. Let $K$ be the Galois closure of $\mathbb{Q}(\alpha)$ with Galois group $G$ and $\bar{\mathbb{Q}}$ be the algebraic closure of $\mathbb{Q}$. In this article, among the other results, we prove the following. {\it If $f\in \bar{\mathbb{Q}}[G]$ is a non-zero element of the group ring $\bar{\mathbb{Q}}[G]$ and $\alpha$ is a given algebraic number such that $f(\alpha^n)$ is a non-zero algebraic integer for infinitely many natural numbers $n$, then $\alpha$ is an algebraic integer.} This result generalizes the result of Polya \cite{polya}, Corvaja and Zannier \cite{corv} and Philippon and Rath \cite{rath}. We also prove the analogue of this result for rational functions with algebraic coefficients. Inspired by a result of B. de Smit \cite{smit}, we prove a finite version of the Polya type result for a binary recurrence sequences of non-zero algebraic numbers. In order to prove these results, we apply the techniques of  Corvaja and Zannier along with the results of Kulkarni {\it et al.}, \cite{kul} which are applications of the Schmidt subspace theorem. 
	\end{abstract}
	\maketitle 
	
	\section{Introduction}
	We deal with the problem of determining whether a given algebraic number $\alpha$ is an algebraic integer under certain conditions. %By definition, we can show that the minimal polynomial of $\alpha$ over $\mathbb{Q}$ has integer coefficients. 
	In 1915, Polya had proved the following statement: \textit{If $\Tr_{\Q(\alpha)/\Q} (\alpha^n)$ are integers for all natural numbers $n$, then $\alpha$ is an algebraic integer}. A proof of the above result, by elementary manipulations, does not seem plausible owing to Newton's identities. In the proof provided by Polya \cite{polya}, he uses Fatou's lemma by considering the generating function of the trace operator (which is a rational function). Alternative proofs were given by H. Lenstra and P. Ponomarev \cite{smit} independently using complementary modules.
	
	\smallskip
	The theorem of Polya doesn't hold when we consider an infinite subset of natural numbers. For example,  for the algebraic number $\alpha = 1/\sqrt{2}$,  we have  $\Tr_{\mathbb{Q}(\sqrt{2})/{\Q}} (1/\sqrt{2})^n$ is always an integer whenever $n\equiv 1 \mod 2$. Hence we need to assume `non-zero integer' condition when we consider the Polya's question for an infinite subset of the natural numbers.    In \cite{corv}, P. Corvaja and U. Zannier, using the subspace theorem, proved the following. {\it Suppose $\alpha$ is an algebraic number and let $E$ be an infinite subset of $\mathbb{N}$. For each $n\in E$, suppose there exists a positive integer $q_n$ such that $\lim_{n\in E} (\log q_n)/n = 0$ and ${\mathrm{Tr}}_{\mathbb{Q}(\alpha)/\mathbb{Q}}(q_n\alpha^n) \in \mathbb{Z}$. Then $\alpha$ is either the $h$-th root of a rational number for some positive integer $h$ or an algebraic integer.}  P. Philippon and P. Rath \cite{rath}  proved a similar result by replacing the integer $q_n$ with a constant which is a non-zero algebraic number.
	\smallskip
	
	A finite version of Polya's question  was refined further by B. de Smit \cite{smit} who explicitly computed the constant $C$ in terms of $\alpha$: {\it If $\Tr_{\Q(\alpha)/\Q} (\alpha^m) \in \mathbb{Z}$ for $1 \le m \le C$, then $\alpha$ is an algebraic integer}. A finite constant is expected because we need to evaluate only finitely many elementary symmetric functions at the Galois conjugates of $\alpha$ to determine whether the given value $\alpha$ is an algebraic integer. The constant $C$ in B. de Smit's result is optimal. % for the trace of powers of algebraic numbers.
	\vspace{.2cm}
	
	\begin{comment}
	Here, we need to ensure that the trace doesn't vanish. The vanishing of the trace function can happen when the conjugates of $\alpha$ differ by a root of unity ({\it i.e.,} $\sigma(\alpha)/\alpha$ is a root of unity), thereby giving rise to vanishing sums of roots of unity. We can view $\Tr(\alpha^n)$ as a linear recurrence sequence, and if it vanishes for infinitely many $n$, then two elements in the recurrence relation differ by the root of unity by appealing to the Skolem-Mahler-Lech Theorem. This theme of vanishing of trace powers has been studied in detail in the paper \cite{rath}.  
	\end{comment}
	\smallskip
	
	In this paper, we study two problems, namely, 
	\begin{enumerate}
		\item We consider the generalised power sums of the form $\lambda_1\alpha_1^\ell+\cdots +\lambda_k\alpha_k^\ell$ where $\lambda_i, \alpha_j$ are algebraic numbers and $\ell\geq 1$ is any integer. If the power sum is an algebraic integer for infinitely many $\ell$'s, then under what conditions, we can conclude $\alpha_j$'s are algebraic integers?
	 As applications of this result (Theorem \ref{TH:ALMOST-VANISHING/ALG-INT} below),  we consider the analogous situation over polynomials, group rings, function fields and a linear combination of trace powers of algebraic numbers.    
		\item Some special cases wherein, we can restrict $\ell$ to a finite (effective) set for a generalised power sums.   
	\end{enumerate}
	The main ideas of our work are coherent with those of \cite{corv}, \cite{smit} and \cite{rath}. 
	
	\section{Our Results}
	Given a set of non-zero algebraic numbers $\alpha_1,\dots,\alpha_k$, we partition them into equivalence classes by the following equivalence relation: 
	\begin{equation} \label{eq2.1}
		\alpha_i \sim \alpha_j \text{ if and only if their ratio is a root of unity}. 
	\end{equation}
	The algebraic numbers $\alpha_1, \ldots, \alpha_k$ are said to be {\it non-degenerate} if they have $k$-equivalence classes. 
	% Written by Abhishek. 
	In general, by the equivalence relation in \eqref{eq2.1},  a tuple $(\alpha_1,\dots,\alpha_k)$ induces a partition on the index set  $\mathcal{I} =\{1,\dots,k\}$,  that is, $\mathcal{I} = \cup_j I_j$, where for each $j$,  the set $\{\alpha_r \ :  \ r\in I_j\}$ is an equivalence class under \eqref{eq2.1}. 
	
	We denote the field of all algebraic numbers by $\overline{\mathbb{Q}}$ and the ring of all algebraic integers by $\overline{\mathbb{Z}}$,   and we set $\zeta_h:=e^{2\pi i /h}$ for an integer $h\geq 2$.   Now we state one of the main theorems as follows. 
	\begin{Th}\label{TH:ALMOST-VANISHING/ALG-INT}
		Let $\mathcal{L}(X_1,\dots,X_k) =\displaystyle \sum_{i=1}^k \lambda_i X_i$ be a linear form with coefficients $\lambda_i$ in $\overline{\mathbb{Q}}^\times$. Let $(\alpha_1, \ldots, \alpha_k) \in \bar{\mathbb{Q}}^{\times k}$ be a given $k$-tuple of algebraic numbers such that  
 		$$
		\mathcal{L}(\alpha_1^n, \ldots, \alpha_k^n) \in \overline{\mathbb{Z}},
		$$
		for $n$ in an infinite set % written by Abhishek 
		$\mathfrak{S} \subset \mathbb{N}$.  Then for each subset $I_j$ of $\mathcal{I} = \{1,\dots, k \}$ corresponding to a equivalence class induced by \eqref{eq2.1},  one of the following holds true: 
		\begin{enumerate}
			\item 
			
			\label{COND:VANISHING} 
			\begin{comment} 
			There exists a positive even integer $h$ and integers $\omega_a$ for $a \in I_j$ satisfying the following: 
			$$
			\sum_{a \in I_j} \lambda_a \zeta_h^{\omega_a n} = 0
			$$
			for all but finitely many $n \in \mathfrak{S}$. Moreover for $a,b \in I_j$, the numbers $\alpha_a, \alpha_b$ and the integers $\omega_a,\omega_b$ satisfy the following relation:
			$$
			\frac{\alpha_b}{\alpha_a} = \zeta_h^{\omega_b - \omega_a}.
			$$
			Thus this condition is equivalent to saying that
			\end{comment} 
			We have $\displaystyle\sum_{a \in I_j} \lambda_a \alpha_a^n = 0$ for all but finitely many values of\textbf{ $n \in \mathfrak{S}$.} 
			\item The numbers $\alpha_i$ are algebraic integers for all $i\in I_j$.
		\end{enumerate}
	\end{Th}
	\begin{comment}
	\begin{Rem}
	In particular, any non-degenerate collection of algebraic numbers $\alpha_1,\dots,\alpha_k$ satisfying the conditions in Theorem \ref{TH:ALMOST-VANISHING/ALG-INT} are algebraic integers.  
	\end{Rem}
	\end{comment}
	Now we proceed to provide some consequences of the above theorem. Theorem \ref{polynomial} and Theorem \ref{groupring} follow immediately from Theorem \ref{TH:ALMOST-VANISHING/ALG-INT}, and the conditions imposed on these theorems are necessary to ensure that a given algebraic number is an algebraic integer.  
	
	\subsection{Polynomial Values : } Here the linear form $\mathcal{L}(X_1,\dots,X_k)$ can be replaced with a polynomial $P(X_1,\dots,X_k)$  satisfying some mild hypotheses. %The following theorem easily follows from Theorem \ref{TH:ALMOST-VANISHING/ALG-INT}.
	\begin{Th}\label{polynomial}
		Let $P(X_1,\dots, X_k)$ be a polynomial with algebraic coefficients 
		%and   for some integer $m$ with $1\leq m\leq k$, 
		and assume that $P(X_1,0,\dots,0, 0)$ is a non-constant polynomial. 
		Let $\alpha_1,\dots,\alpha_k$ be multiplicatively independent non-zero algebraic numbers such that   $P(\alpha_1^n,\dots,\alpha_k^n) \in \overline{\Z}$ for infinitely many positive integers $n$. Then the number $\alpha_1$ is an algebraic integer.
	\end{Th}
	
	%We have the following immediate corollary. 
	
	%\begin{Cor}\label{poly-coro}
	%Let $P(X_1,\dots, X_k)$ be a polynomial with algebraic coefficients and    for each integer $m$ with $1\leq m\leq k$, assume that $P(0,\dots,0, X_m,0,\dots 0)$ is a non-constant polynomial.
	%Let $\alpha_1,\dots,\alpha_k$ be multiplicatively independent non-zero algebraic numbers such that   $P(\alpha_1^n,\dots,\alpha_k^n) \in \overline{\Z}$ for infinitely many positive integers $n$. Then the number $\alpha_m$ is an algebraic integer for each integer $m$.
	%\end{Cor}
	
	\begin{Rem}
		Theorem \ref{polynomial} is no longer true if we remove the condition on $P(X_1, 0, \ldots, 0)$ is a non-constant polynomial. %(X_1,0,\dots 0)$ is a non-constant polynomial. 
		 For instance, let $P(X,Y) = XY + 1$ and fix a prime $p\neq 2$. Choose $\alpha_1= 1/p$ and $\alpha_2 = 2p$ and  note that $P(\alpha_1^n,\alpha_2^n) = 2^n + 1 \in \Z$ for all natural numbers $n$. 	
		 \end{Rem}
	
	\subsection{Group Rings}  Let $\alpha$ be a non-zero algebraic number,  $K_\alpha$ be the Galois closure of $\Q(\alpha)$ and  $G_\alpha = \Gal{K_\alpha/\mathbb{Q}}$ be the Galois group. We identify the index set $\mathcal{I}$ of Theorem \ref{TH:ALMOST-VANISHING/ALG-INT} as the elements of $G_\alpha$, for convenience, and consider the group ring $\bar{\mathbb{Q}}[G_\alpha]$. 
	\begin{Th}\label{groupring}
		Let $f \in \overline{\Q}[G_\alpha]$ be a non-zero element   such that  that  $f(\alpha^n) \in \overline{\Z}\backslash \{0\}$ for $n$ in an infinite subset $\mathfrak{S}\subset\N$.  Then $\alpha$ is an algebraic integer.
	\end{Th}
	
	\begin{Rem}
		In particular, Theorem \ref{groupring} proves the following: {\it If $\sigma(\alpha^n) - \alpha^n$ is a non-zero algebraic integer for infinitely many $n$, then $\alpha$ is an algebraic integer}. However, the same conclusion is not possible by applying the trace operator because $\Tr_{K/\Q} ( \sigma(\alpha^n) - \alpha^n) = 0$.   One notes that this fact can be obtained by applying the number field version of Ridout's theorem, see for instance,   Corollary 1.2 in Chapter 7 of S. Lang \cite{lang} (or Theorem D.2.1 in \cite{Hindry}).	\end{Rem}
	
	\subsection{Action under the trace map}
	
	Let $\alpha$ be a non-zero algebraic number of degree $d$ and let $\alpha = \alpha_1, \ldots, \alpha_d$ be all the Galois conjugates of $\alpha$. If $P(X)\in \mathbb{Q}[X]$ is a non-zero polynomial of $\alpha$, then 
	$$
	{\mathrm{Tr}}_{K/\mathbb{Q}}(P(\alpha^n)) =Q(\alpha_1^n, \ldots, \alpha_d^n)
	$$
	where $Q(X_1, \ldots, X_d) = P(X_1)+\cdots+P(X_d)$. Though the polynomial $Q$ satisfies the hypothesis of Theorem \ref{polynomial}, the numbers $\alpha_1, \alpha_2, \ldots, \alpha_d$ need not be multiplicatively independent. However, we have the following theorem for the trace operator.

	\begin{Th}\label{TH:PHILI-RATH-POLY-ANALOG}
		Let $\alpha$ be a non-zero algebraic number and let $P(X)=\lambda_D X^D+\cdots+\lambda_0\in \bar{\mathbb{Q}}[X]$ be a non-constant polynomial of degree $D$ and let $L = \mathbb{Q}(\lambda_0, \ldots, \lambda_D, \alpha)$.   If  ${\mathrm{Tr}}_{L/\mathbb{Q}}(P(\alpha^n))\in \mathbb{Z}$  for each $n$ in  an infinite set $\mathfrak{S}$ of natural numbers, then either  $\alpha$ is an algebraic integer or for each $i= 1,2,\ldots, D$, we have  ${\mathrm{Tr}}_{L/\mathbb{Q}}(\lambda_i \alpha^{i n})=0$ for all but finitely many $n\in \mathfrak{S}$.    	\end{Th}

	When we consider a multivariable generalisation of the above theorem, we may have that the trace operator vanishes for a subsum for trivial reasons. For simplicity,  we consider a linear form in several variables.
	\begin{Th}\label{trace}
		Suppose $\alpha_1,\dots,\alpha_k,\lambda_1,\dots \lambda_k$ be distinct non-zero algebraic numbers. Let $L= \Q(\alpha_i,\lambda_i~|~ 1 \le i \le k)$,  $K$ be its Galois closure and $h$ be the order of the torsion subgroup of $K^\times$ over $\Q$.  Suppose 
		$
		\mathrm{Tr}_{L/\Q}(\lambda_1\alpha_1^n+\cdots+\lambda_k\alpha^n_k) \in \Z
		$
		for $n$ in an infinite subset $\mathfrak{S}\subset\N$.  If $\alpha_1$ is not an algebraic integer, then there exists an integer $a \in \{0,\dots,h-1\}$ and a subset $I$ of $\{1,\dots,k\}$ containing $1$ such that 
		$$
		\mathrm{Tr}_{L/\Q} \left(\sum_{i\in {I}} \lambda_i \alpha_i^a\right) =0
		$$
		where for each $i \in {I}$, there exists $\sigma_i \in \mathrm{Gal}(K/\Q)$, the Galois group of $K$,  such that $\sigma_i(\alpha_i)/\alpha_1$ is a root of unity.
	\end{Th}
	
	We have the following interesting corollary as a consequence of this result. 
	\begin{Cor}\label{COR:MULTISUM-TRACE-INTEGERS}
		Let $\alpha_1,\ldots,\alpha_k, \lambda_1,\ldots,\lambda_k$ be non-zero algebraic numbers. Let $L = \mathbb{Q}(\alpha_i,\lambda_i~|~ 1 \le i \le k)$, $K$ be its Galois closure and $h$ be the order of the torsion subgroup of the multiplicative group $K^\times$.   Suppose  that no subsum of $\lambda_1\alpha_1^a+\cdots +\lambda_k \alpha^a_k$ vanishes, under the trace map,   for each $a\in \{0, 1, \ldots,h-1\}$. If ${\mathrm{Tr}}_{L/\mathbb{Q}}(\lambda_1\alpha_1^n+\cdots +\lambda_k \alpha^n_k)\in\mathbb{Z}$ for infinitely many natural numbers $n$, then each $\alpha_i$ is an algebraic integer for all $i=1,\ldots,k$.   
	\end{Cor}
	 The above corollary can be considered as a multidimensional generalization of a result of P. Philippon and P. Rath in \cite{rath}.
	
	\subsection{Determining the nature of rational functions}
	We consider a function field (of characteristic $0$) analogue of Theorem \ref{TH:ALMOST-VANISHING/ALG-INT} in the simplest setting after imposing some additional restrictions.
	 
	\begin{Th}\label{rationalfunction}
		Let $f_1(X),\dots,f_k(X)$ be non-constant rational functions with algebraic coefficients and $\lambda_1,\dots, \lambda_k$ be non-zero algebraic numbers.  Assume that the ratio $f_i(X)/f_j(X)$ is not a constant function for each $i \neq j$.  If 
		\begin{equation} \label{eq2.2}
			\sum_{i=1}^k \lambda_i (f_i(X))^n \in \overline{\Z}[X] 
		\end{equation}
		for $n$ in an infinite subset $\mathfrak{S}\subset \mathbb{N}$, then each $f_i(X) \in \overline{\Z}[X]$.
	\end{Th}
	
	%\begin{Rem}
	%For the second part, to conclude that $f_i(x) \in \mathcal{O}_K[x]$ given that $f_i(x)\in K[x]$ it is necessary to take a ramified extension. This is because we have integer valued polynomials which are not in $\mathcal{O}_K[x]$.  For example, if we take $f(x)= x(x-1)/2$, then $f(\Z) \subseteq \Z$. 
	%\end{Rem}
	
	It might be possible to show each of the functions $f_i(X)$ are polynomial functions without appealing to the subspace theorem. However, the approach here is to deduce the nature of rational function via its specialisations. %After the application of the subspace theorem, we have only used valuation arguments throughout the proof of this  theorem.
	\smallskip
	
	\subsection{Determining algebraic integers in finite iteration}
	% WHole structure modified by Abhishek.
	In the work of B. de Smit \cite{smit}, a finite bound on $\ell$ in $\Tr_{\Q(\alpha)/\Q} (\alpha^\ell)$ was given in order to determine whether $\alpha$ is an algebraic integer. We provide the following generalizations: 
	\begin{Th}\label{thm1.1}
		Let $\alpha_1$ be a nonzero algebraic number  and $\alpha_2, \ldots, \alpha_k$ be  all the other Galois conjugates of $\alpha_1$ for some integer $k\geq 2$.   Let $K$ be the Galois closure of $\mathbb{Q}(\alpha_1)$  and $[K:\mathbb{Q}] = d\geq k$. For any integers $b_1, \ldots, b_k$ (not necessarily distinct) such that $b_1+\cdots+b_k = n\ne 0$,  if $\displaystyle{\mathrm{Tr}}_{K/\mathbb{Q}}(b_1\alpha_1^j+\cdots+b_k\alpha_k^j) \in \mathbb{Z}$ for all $j = 1, 2, \ldots, d+d[\log_2(nd)]+1$, then $\alpha_1$ is an algebraic integer and so is $\alpha_j$ for each $j \geq 2$.
	\end{Th}
	We now look at a multidimensional analogue of B. de Smit's result more generally, not just for trace operators, but for a  given linear recurrence sequence. Let $\alpha_i\in \bQ^\times$ be distinct algebraic numbers, $\lambda_i \in \bQ^\times$ for all $1 \le i \le k$, $K = \Q(\alpha_1,\dots,\alpha_k,\lambda_1,\dots \lambda_k)$ and $\mathcal{O}_K$ be its ring of integers.  %and let $m_i:=\sum_{j=1}^k \lambda_j \alpha_j^i$. 
	We ask the following question:
	\begin{Ques}
		Does there exist a constant $C$ (depending on  $\alpha_i$ and $\lambda_i$) such that  
		\[ \lambda_1\alpha_1^\ell+\cdots+\lambda_k\alpha_k^\ell \in \mathcal{O}_K \text { for all } 1\le \ell \le C \implies \alpha_j \in \mathcal{O}_K \textrm{ for all }1 \le j \le k? \]
	\end{Ques}
	When $k=1$, one can answer the above question positively with $C = 1 + \max_\mathfrak{p} |v_\mathfrak{p}(\lambda_1)|$ where $\mathfrak{p}$ runs through the non-zero prime ideals in $\mathcal{O}_K$ and $v_\mathfrak{p}$ denotes the valuation at $\mathfrak{p}$.  Here, we answer this question when $k=2$ and highlight some difficulties in generalising this proof for arbitrary $k$ in Remark \ref{REM:TH-finite-difficulties}. 
	%We state the following theorem : 
	\begin{Th}\label{finite}
		Let $\alpha_1, \alpha_2, \lambda_1,\lambda_2$ be given non-zero algebraic numbers. Let $K = \Q(\lambda_1,\lambda_2, \alpha_1, \alpha_2)$ and let 
		$$
		C = 2+ \left\lceil\frac{1}{2}\max_{\mathfrak{P}}\left|v_\mathfrak{P}(\lambda_1\lambda_2(\alpha_1-\alpha_2)^2)\right|\right\rceil+\max_{\mathfrak{P}}\left|v_\mathfrak{P}(\lambda_1\lambda_2)\right|
		$$
		where $\mathfrak{P}$ runs through the non-zero prime ideals in $\mathcal{O}_K$. 
				If 
		$$
		\lambda_1\alpha_1^\ell+\lambda_2\alpha_2^\ell \in \mathcal{O}_K \text{ for all } 1\le \ell \le C,
		$$  
		then  $\alpha_1,\alpha_2 \in \mathcal{O}_K.$
	\end{Th}
	\section{Preliminaries}
	In this section, we state the propositions/theorems required for the proof of our theorems. We also need some results which are applications of  the Schmidt Subspace Theorem,  formulated by Evertse and Schlickewei.  For a reference,  see (\cite[Chapter 7]{bomb},  \cite[Chapter V, Theorem $1D^\prime$]{schmidt}, and  \cite[Page 16, Theorem II.2]{zannier}).
	
	%%%%%%%%%%%%%%%%%%%%%%%%%%%%%%%%%%%%%%%%%%%(written by Veekesh)
	Let $K$ be a number field which is a Galois extension over $\mathbb{Q}$. Let $M_K$ be the set of all places on $K$ and $M_\infty$ be the set of all archimedean places on $K$. 
	For each place $w\in M_K$, let $K_w$ denote the completion of the number field $K$ with respect to $w$ and $d(w)=[K_w:\mathbb{Q}_\mathit{v}]$, where $\mathit{v}$ is the restriction of $w$ to $\mathbb{Q}$. 
	For  every $w\in M_K$ whose restriction on $\mathbb{Q}$ is $v$ and $\alpha\in K$, we define the  normalized  absolute value $|\cdot|_w$ as follows:
	\begin{equation}\label{eq3}
		%\tag{3.1}
		|\alpha|_w:=|\mbox{Norm}_{K_w/\mathbb{Q}_v}(\alpha)|^{\frac{1}{[K:\mathbb{Q}]}}_v.
	\end{equation}
	Indeed if $w\in M_\infty$, then there exists an automorphism $\sigma\in\mbox{Gal}(K/\mathbb{Q})$ of $K$ such that for all $x\in K$, 
	$$
	|x|_w=|\sigma(x)|^{d(K)/[K:\mathbb{Q}]},
	$$
	where $d(K) =1$ if $K\subset \mathbb{R}$ and $d(K) = 2$  otherwise.  Note that since $K$ is Galois over $\mathbb{Q}$,  the embeddings are either totally real or totally complex, and hence the function $d(K)$ is constant.  Thus, under the definition \eqref{eq3},  the product formula   $\displaystyle\prod_{\omega\in M_K}|x|_\omega=1$ holds true for any  $x\in K^\times$.
	%%%%%%%%%%%%%%%%%%%%%%%%%%%%%%%%%%%%%%%%%%%%
	%Let $K\subset \mathbb{C}$ be a number field which is Galois over $\mathbb{Q}$ with Galois group $\mathrm{Gal}(K/\mathbb{Q})$. Let $M_K$ be the set of all inequivalent places of $K$ and $M_\infty$ be the set of all archimedean places of $K$. For each place $v\in M_K$, we denote $|\cdot |_v$ the absolute value corresponding to $v$, normalized with respect to $K$. Indeed if $v\in M_\infty$, then there exists an automorphism $\sigma\in\mbox{Gal}(K/\mathbb{Q})$ of $K$ such that for all $x\in K$, 
	%\begin{equation*}
	%|x|_v=|\sigma(x)|^{d(\sigma)/[K:\mathbb{Q}]},
	%\end{equation*}
	%where $d(\sigma) =1$, if $\sigma(K) = K\subset \mathbb{R}$; and $d(\sigma) = 2$ otherwise. Note that since $K$ is Galois over $\mathbb{Q}$, the function $d(\sigma)$ is constant. Non-archimedean absolute values are normalized accordingly so that  the product formula   $\displaystyle\prod_{\omega\in M_K}|x|_\omega=1$ holds true for any  $x\in K^\times$. 
	For a non zero vector $\mathbf{x} = (x_1, \ldots, x_n)\in K^n$, the {\it projective height},  $H(\mathbf{x})$, is defined by 
	$$
	H(\mathbf{x})=\prod_{\omega\in M_K}\mbox{max}\{|x_1|_\omega,\ldots,|x_n|_\omega\}. 
	$$
We require the following results from \cite{kul}. 	
% We shall use the following theorem, which is an application of the subspace theorem. % likely well known before A. Kulkarni, N. Mavraki and K. D. Nguyen } 
	\begin{Prop}\label{PROP:KUL-FINITE-VANISHING}{\rm(A. Kulkarni, N. Mavraki and K. D. Nguyen \cite{kul})}~ 
		Let $\alpha_1,\ldots,\alpha_k$  be non-degenerate  non-zero algebraic numbers and let $\lambda_1,\ldots, \lambda_k$ be non-zero algebraic numbers.  Then there are at most  finitely many natural numbers $n$ satisfying 
		$$
		\lambda_1\alpha^n_1+\cdots+\lambda_k\alpha^n_k=0.
		$$
	\end{Prop}
	%We shall use the following theorem as an alternative of the subspace theorem (though derived from it): 
	\begin{Th}\label{TH:KUL-SUBSPACE}{\rm(A. Kulkarni, N. Mavraki and K. D. Nguyen \cite{kul})}~    
		Let $K$ be a number field which is  Galois  over $\mathbb{Q}$  and $S$  be a finite set of places, containing all  the archimedean places.  Let 
		$\lambda_1,\ldots,\lambda_k$  be non-zero elements of $K$.  Let  $\varepsilon>0$  be a positive real number and $\omega\in S$  be a distinguished place.  Let $\mathfrak{E}$ be the set of  all  $(u_1,\ldots,u_k) \in (\mathcal{O}_S^\times)^k$ which satisfy  the inequality 
		\begin{equation}\label{eq3.1}
			0<\left|\sum_{j=1}^k\lambda_j u_j\right|_\omega\leq \frac{\max\{|u_1|_\omega,\ldots,|u_k|_\omega\}}{H(u_1,\ldots,u_k,1)^\varepsilon},
		\end{equation}
		where $\mathcal{O}_S^\times$ is the ring of $S$-units in $K$.  If $\mathfrak{E}$ is an infinite set, then there exist $c_1, \ldots, c_k\in K$, not all zero, such that    
		$$
		c_1  u_1+\cdots+c_k  u_k=0 
		$$ 
		holds  true for infinitely many elements $(u_1, \ldots, u_k)$ of $\mathfrak{E}$.
	\end{Th}
	We prove the following proposition regarding the elements of Galois group fixing the equivalence class of a non zero algebraic element $\alpha$ (given by \eqref{eq2.1}). This is of independent interest.
	\begin{Prop}\label{PROP:SUBGROUP-VANISHING}
		Let $\alpha$ be a non-zero  algebraic number, and $K$ be any Galois extension of $\Q$ containing $\alpha$ with its Galois group $G= \mathrm{Gal}({K/\Q})$. Let 
		$$
		H = \left\{\sigma \in G \ :  \ \sigma(\alpha) \sim \alpha\right\}$$
		be a subset of $G$.  Then the following statements are true; 
		\begin{enumerate}
			\item[(1)] $H$ is a subgroup of $G$. 
			\item[(2)] For any given $\tau \in G$, we have, $\{ \sigma \in G ~|~ \sigma(\tau(\alpha)) \sim \tau(\alpha) \} = \tau H \tau^{-1}$.
		\end{enumerate}
	\end{Prop}
	\begin{proof}
		If $\sigma_a \in H$ then we have $\sigma_a(\alpha) = \zeta_h^{w_a} \alpha$ for some integer $w_a$. Since $G$ permutes the roots of unity, so does $H$. This shows that $H$ is closed under inversion and composition of automorphisms in $G$. Hence $H$ is a subgroup of $G$. This proves (1). 
		
		Let $\sigma \in H$ be any element. Then we have 
		\[\sigma(\alpha) \sim \alpha \Leftrightarrow \tau(\sigma(\alpha)) \sim \tau(\alpha) \Leftrightarrow (\tau\sigma\tau^{-1})(\tau\alpha) \sim \tau(\alpha), \]
		proving the second statement.
		%Suppose $f(\alpha^n) = 0 = \displaystyle\sum_{i} \lambda_i \sigma_i(\alpha)^n$ for infinitely many values of $n$.   By Proposition \ref{PROP:KUL-FINITE-VANISHING}, we note that the tuple $(\sigma_i(\alpha))_i$ is a degenerate tuple. As in Theorem \ref{TH:ALMOST-VANISHING/ALG-INT}, we partition the index set by the equivalence relation \eqref{eq2.1}. We write $G = H \cup_j I_j$. For $\sigma_a,\sigma_b \in I_j$ we have $\sigma_b(\alpha) = \zeta_h^{w_b}\sigma_a(\alpha)$. Therefore $\sigma_b^{-1} \sigma_a \in H$. This proves that $I_j = \sigma_b H$.
	\end{proof}
	We shall state the following basic and well-known lemma which roughly says `integrality' is a local phenomenon. 
	
	\begin{Lemma}\label{lem0}
		Let $\alpha\in \bar{\mathbb{Q}}$ be an algebraic number. Then $\alpha$ is an algebraic integer if and only if $\alpha$ is integral over $\mathbb{Z}_p$ for every prime number $p$ where $\mathbb{Z}_p$ is the ring of $p$-adic integers.
	\end{Lemma}
	
	The following lemma is also basic and well-known and hence we omit the proof here. 
	
	\begin{Lemma} \label{lem1} 
		Let $K$ be a finite extension over $\mathbb{Q}_p$ of degree $d$ where $\mathbb{Q}_p$ is the field of $p$-adic numbers.  Let $\alpha\in K$  be an element such that $\alpha\not\in \mathcal{O}_K$, the local ring of $K$.  Then $\beta := \alpha^{-1} \in \mathcal{O}_K$. Furthermore,  if the characteristic polynomial of $\beta$  is 
		$$f_{K|\mathbb{Q}_p}(x) = a_dx^d+a_{d-1}x^{d-1}+\cdots+a_1x+a_0 \in \mathbb{Z}_p[x],
		$$ 
		then $a_d = 1$ and  $a_i\in p\mathbb{Z}_p$ for all $i = 0, 1,\ldots, d-1$ where $p\mathbb{Z}_p$ is the unique maximal ideal of $\mathbb{Z}_p$.
	\end{Lemma}
	
	We need the following lemma in the proof of Theorem  \ref{thm1.1}. 
	
	\begin{Lemma}\label{lem2}
		Let $p$ be a given prime number and let $K$ be a finite Galois extension over $\mathbb{Q}_p$ of degree $d$. Let $\alpha\in K$  be an element such that $\alpha\not\in \mathcal{O}_K$  and $\beta$ be a Galois conjugate of $\alpha$. Let $b_1$ and $b_2$ be given $p$-adic integers such that $b_1+b_2 \ne 0$.
		Then for each integer $\ell \geq 0$,   there exists an integer $i \in \{1,2,\ldots, d\}$ such that $b_1\alpha^{\ell+i}+b_2\beta^{\ell+i}\ne 0$. 
	\end{Lemma}
	\begin{proof}
		Since $\alpha\not\in\mathcal{O}_K$, by Lemma \ref{lem1}, it is clear that $\alpha^{-1}\in\mathcal{O}_K$ and let $f(x) =  x^d+a_{d-1}x^{d-1}+\cdots+a_1x+a_0  \in \mathbb{Z}_p[x]$ be the characteristic polynomial of $\alpha^{-1}$. Since $\beta$ is a Galois conjugate of $\alpha$, it is clear that  $f(\alpha^{-1}) = 0 = f(\beta^{-1})$.  We see that for every integer $m\geq 0$, we have 
		$$
		\alpha^{m} = - a_{d-1}\alpha^{m+1}-\cdots  - a_0\alpha^{m+d} \mbox{ and }  \beta^{m} = - a_{d-1}\beta^{m+1}-\cdots  - a_0\beta^{m+d}.
		$$
		Therefore,  we get
		\begin{equation}\label{eq3.2}
			b_1\alpha^{m} +b_2\beta^m = -a_{d-1}(b_1\alpha^{m+1}+b_2\beta^{m+1}) - \cdots  -a_0(b_1\alpha^{m+d} +b_2\beta^{m+d}).
		\end{equation}
		
		Now we prove the assertion by induction on $\ell$. 
		
		We put $m = 0$ in \eqref{eq3.2}, we get 
		$$
		0\ne b_1+b_2 = -a_{d-1}(b_1\alpha^{1}+b_2\beta^{1}) - \cdots  -a_0(b_1\alpha^{d} +b_2\beta^{d})  
		$$
		which implies the assertion when $\ell = 0$.   Assume the assertion is true for some integer $\ell > 0$. That is,  there exists an integer $i\in \{1,2,\ldots, d\}$ such that $b_1\alpha^{\ell+i} +b_2\beta^{\ell+i} \ne 0$. If $i\geq 2$, then $\ell + i = \ell +1 + j$ for some integer $j \in \{1,2,\ldots, d\}$ and then we are done. Hence we assume that $i = 1$. That is, $b_1\alpha^{\ell+1}+b_2\beta^{\ell+1} \ne 0$. Now, put $m = \ell +1$ in \eqref{eq3.2} to get the assertion in this case as well.  Hence the lemma.
	\end{proof}
	
	We conclude this section by discussing a proof of Fatou's Lemma. This was proved by Pisot \cite{PISOT} for number fields and the proof was adapted based on Fatou's work \cite{FATOU}.  We provide a slightly different proof along the lines mentioned in B. de Smit \cite{smit}. Before proceeding further, for a non-zero algebraic number $\alpha$, we define the denominator of $\alpha$ (denoted by $\textrm{den}(\alpha)$) to be the smallest positive integer $n$ such that $n\alpha$ is an algebraic integer. 
	\begin{Prop}[Fatou's lemma]
		Let $K$ be a number field and $f(X) \in K(X) \cap \mathcal{O}_K[[X]]$ such that $f(X) = g(X)/h(X)$ where $g(X),h(X) \in K[X]$ are co prime in $K[X]$ and $h(0) = 1$. Then $g(X), h(X) \in \mathcal{O}_K[X]$.
	\end{Prop}
	\begin{proof}
		Given that $f(X) = g(X)/h(X)$ where $g(X),h(X) \in K[x]$ are co prime and $h(0)=1$. Then there exist two polynomials $r(X),s(X) \in \mathcal{O}_K[X]$ such that $r(X)g(X)+ s(X)h(X) = c$ for some non zero constant $c \in \mathcal{O}_K$. Therefore, we have 
		\begin{equation}\label{eq3.3}
			\frac{c}{h(X)} \in \mathcal{O}_K[[X]].
		\end{equation}
		Now let $h(X) = \displaystyle\prod_{i=1}^k (1 - a_i X)^{c_i}$ in $\overline{\Q}[X]$ for some integers $c_i \ge 1$ and some distinct $a_i \in \overline{\Q}$. For each $i$ with $1\leq i\leq k$, multiplying \eqref{eq3.3}, by the factor 
		$$
		\left[\prod_{\substack{j=1\\j \neq i}}^k \textrm{den}(a_j)^{c_j} (1 - a_j X)^{c_j}\right] \left[\textrm{den}(a_i)^{c_i-1} (1-a_i X)^{c_i-1}\right],
		$$
		we define a new formal power series $\widetilde{f_i}(X):= c'(1-a_iX)^{-1}$ for some $c'$ depending on $\textrm{den}({a_j})$ and $c_j$ for all $ 1 \le j \le k$. Since we are multiplying by a polynomial whose coefficients are  algebraic integers, we get  $\widetilde{f_i}(X)\in \mathcal{O}_L[[X]]$ for some finite extension $L$ over $K$ in which  $h(X)$ splits completely. Therefore for any non zero prime ideal $\mathfrak{P}$ in $\mathcal{O}_L$, we have 
		\[ v_\mathfrak{P}(c'a_i^n) \ge 0 \text{ for all } n \implies v_\mathfrak{P}(a_i) \ge 0 \]
and consequently, we get $a_i$ is an algebraic integer.  Therefore $h(X) \in K[X] \cap \mathcal{O}_L[X] = \mathcal{O}_K[X]$ and hence $g(X) = f(X)h(X) \in \mathcal{O}_K[X]$ as both $f(X),h(X)$ take values in $\mathcal{O}_K$.
	\end{proof}
	%\begin{Rem}
	%It is possible to simultaneously prove that the numbers $a_i$ are algebraic integers by considering the function $c'/\prod_{i=1}^k (1-a_ix)$ and apply Theorem \ref{TH:ALMOST-VANISHING/ALG-INT} for different linear forms $\mathcal{L}_m$.   
	%\end{Rem}
	
	\section{Proofs of   Theorems 2.1 to \ref{rationalfunction}}
	\noindent{\bf Proof of Theorem \ref{TH:ALMOST-VANISHING/ALG-INT}.}
	Let $K = \Q(\alpha_1,\dots, \alpha_k)$ be a number field and $h$ be the order of the torsion subgroup of $K^\times$. We partition $\{ \alpha_1,\dots,\alpha_k \} = \displaystyle\bigcup_{l=1}^d \{ \alpha_a~\colon ~a \in I_l \}$ by the equivalence relation in \eqref{eq2.1}. For each index set $I_l$, let $\beta_l$ be a representative of the corresponding equivalence class $\{\alpha_a~\colon~a \in I_l \}$, and for $a \in I_l$, let $\alpha_a = \beta_l \zeta_h^{\omega_{a,l}}$ for some integer $\omega_{a,l}$.  The numbers $\beta_1, \dots, \beta_d$ are representatives (elements) of disjoint equivalence classes and hence the tuple $(\beta_1,\dots,\beta_d)$ is non-degenerate. 
	
	\bigskip
	
Let $l$ be a fixed natural number with $1\leq l\leq d$  and assume that assertion \eqref{COND:VANISHING} is not true for $I_l$.  Then, for infinitely many $n \in \mathfrak{S}$,  the sum $\displaystyle\sum_{a \in I_l} \lambda_a \alpha_a^n \neq 0$.  In particular,  there exists  an infinite subset  $\mathfrak{S}_m := \{n\in \mathfrak{S} \ : \ n\equiv m\mod h\}$ for some $m\in \{0,1,\ldots, h-1\}$ such that  $\displaystyle\sum_{a \in I_l} \lambda_a \alpha_a^n \neq 0$ for each $n\in \mathfrak{S}_m$.   %We restrict $\mathfrak{S}$ to $\mathfrak{S}_m$. and construct a new linear form consisting of degenerate elements with the help of the equivalence relation \eqref{eq2.1}. 
	%We write $\displaystyle\mathfrak{S} = \bigcup_m \mathfrak{S_m} \cup \mathfrak{T}$, where each $\mathfrak{S_m}$ is a subset of $\mathfrak{S}$ consisting of positive integers $n'$ such that $n' \equiv m \mod h$ and $\mathfrak{T}$ is a finite set.  
	%Since $\mathfrak{S}$ is infinite, we note that there exists an integer $m$ such that for infinitely many integers $n$ of $\mathfrak{S}$ we have that $n \equiv m \mod{h}$. We denote this set to be $\mathfrak{S}_m$.
	%This ensures there exists at least one $m$ for which $\mathfrak{S}_m$ is infinite. 
	%We partition $ where each $S_j$ denotes an equivalence class given by \eqref{eq2.1}. %Furthermore, for each $S_j$, we correspondingly define an index set $I_j$ consisting of the indices $a$ such that $\alpha_a \in S_j$. We select $\beta_j \in S_j$ to be a representative of this class.  
	Therefore, for each $n \in \mathfrak{S}_m$, we have 
	\[
	\mathcal{L}(\alpha_1^n,\dots,\alpha_k^n) = \sum_{i=1}^k \lambda_i \alpha_i^n = \sum_{r=1}^d \sum_{a \in I_r} \lambda_a \alpha_{a}^n = \sum_{r=1}^d \sum_{a \in I_r} \lambda_a  \zeta_h^{n\omega_{a,r}} \beta_r^n = \sum_{r=1}^d \kappa_{r,m} \beta_r^n =:\mathcal{L}_m(\beta_1^n,\dots,\beta_d^n) , 
	\]
	where $\kappa_{r,m}: =\displaystyle \sum_{a \in I_r} \lambda_{a}  \zeta_h^{m\omega_{a,r}}$, as $n \equiv m \mod h$. 
	Since $\displaystyle\kappa_{l,m} \beta_l^n = \sum_{a \in I_l} \lambda_a \alpha_a^n$, we have $\kappa_{l,m} \neq 0$.  Furthermore, since $(\beta_1, \ldots, \beta_d)$ is a non-degenerate tuple, by Proposition \ref{PROP:KUL-FINITE-VANISHING}, $\mathcal{L}(\beta_1^n,\dots,\beta_d^n) \in \overline{\mathbb{Z}}\backslash\{0\}$ for all but finitely many values of $n \in \mathfrak{S}_m$. We remove the finite numbers $n \in   \mathfrak{S}_m$ for which $\mathcal{L}(\beta_1^n,\dots,\beta_d^n) = 0$. 
	
	%Therefore, for each residue class $m \mod{h}$ which appears infinitely many times in $\mathfrak{S}$, we have a new linear form 
	%$$
	%\mathcal{L}_m(X_1,\dots, X_d) = \sum_{j=1}^d \kappa_{j,m} X_j
	%$$
	%with $d$ depending on $\alpha_i, \lambda_i$ and $m$. \smallskip
	
	%We arrive at a reduced linear form $\displaystyle\mathcal{L}_m(X_1,\dots, X_d) = \sum_{i=1}^d \kappa_{i,m} X_i$  we assume that each of the coefficients is non-zero and in particular  \smallskip 
	\begin{comment}
	We have reduced our question to the following : 
	$$
	\mathcal{L}_m(\overline{\beta}^n)  \in \overline{\mathbb{Z}} \backslash \{0 \},
	$$
	where $\overline{\beta} = ( \beta_1,\dots,  \beta_d)$.   
	\end{comment}
	\smallskip
	
	We proceed to prove that $\beta_l$ is an algebraic integer  via contradiction.
	
	 Let $S$ be a suitable finite subset of $M_K$ containing all the archimedean places such that $\beta_j$ is an $S$-unit for each $j = 1,2, \ldots, d$. Assume that $\beta_l$ is not an algebraic integer. Then there exists a finite place $\omega\in S$ such that $|\beta_l|_\omega>1$. Choose $\varepsilon >0$ such that $\varepsilon < \displaystyle\frac{\log |\beta_{l}|_\omega}{\log H(\beta_1, \ldots, \beta_{d}, 1)}$. For all $n \in \mathfrak{S}_m$, we have $|\beta_{l}|_\omega^{n} H(\beta_{1}^{n}, \ldots, \beta_d^{n}, 1)^{-\varepsilon}>1$ because  by the choice of $\varepsilon$ and the height inequality $$H(x_1y_1:\ldots, x_dy_d:1) \le H(x_1:\ldots,x_d:1)H(y_1:\ldots,y_d:1).
	$$ Since $\mathfrak{S}_m$ is an infinite set, Theorem \ref{TH:KUL-SUBSPACE} (applied to the linear form $\mathcal{L}_m$ consisting of the sub-tuple $(\dots,\beta_l,\dots)$ for which $\kappa_{l,m} \neq 0$) asserts that  
	%Since $\beta_{i} \in \mathcal{O}_S^\times$ and $\eta_{i,j} \in K \backslash\{0\}$, we see that $\rho_j(\beta_i^h) \in \mathcal{O}_S^\times$.  Since $\mathfrak{S}$ is an infinite set,   by \eqref{eq2} and Theorem \ref{lem2.1}, we obtain a non-trivial relation 
	\begin{equation*}
		\sum_{i=1}^d b_i\beta_i^{n}= 0
	\end{equation*}
	holds true for infinitely many natural numbers  $n \in \mathfrak{S}_m$, where $b_i \in K$ and not all zero. This is a contradiction to Proposition \ref{PROP:KUL-FINITE-VANISHING}.  % Since $(\beta_1,\ldots,\beta_d)$ is a  non-degenerate tuple of algebraic numbers,  this is impossible by  Proposition \ref{PROP:KUL-FINITE-VANISHING}.    Thus we conclude that $\beta_j$ is an algebraic integer. $\hfill\Box$
	%$$ by Proposition which is a contradiction as the numbers $\beta_i$ for $1 \le i \le d$ are non-degenerate. 
	Therefore   $\beta_l$ is an algebraic integer. $\hfill\Box$
	\vspace{.3cm}
	
	\noindent{\bf Proof of Theorem \ref{polynomial}.} Given polynomial $P(X_1, \ldots, X_k)$, we write a typical monomial in $P$ as $\displaystyle \prod_{j=1}^k X_j^{b_j} = X_{\mathbf{i}}$ where $\mathbf{i} = (b_1, \ldots, b_k)$. Also, we write the coefficient of the monomial $X_{\mathbf{i}}$ in $P$ as $a_{\mathbf{i}}$.  Let  $\mathcal{I} = \left\{(b_1,\ldots, b_k) \ : \mathbf{i} = (b_1, \ldots, b_k) \mbox{ and } X_{\mathbf{i}} \mbox{ appears in } P  \right\}$ be the index set.     
	
	\bigskip
	 Now, consider the linear form
	  $ \mathcal{L}((Y_{\mathbf{i}})_{\mathbf{i} \in \mathcal{I}}) = \displaystyle\sum_{\mathbf{i} \in \mathcal{I}}a_{\mathbf{i}}Y_{\mathbf{i}}.$  Therefore, by setting $\alpha_{\mathbf{i}} :=\displaystyle \prod_{j=1}^k \alpha_j^{b_j}$ for each $\mathbf{i} = (b_1, \ldots, b_k) \in \mathcal{I}$,  we see that $P(\alpha_1, \ldots, \alpha_k) = \mathcal{L}((\alpha_{\mathbf{i}})_{\mathbf{i} \in \mathcal{I}})$. 
 
 \bigskip
	
By hypothesis, $\alpha_1, \ldots, \alpha_k$ are multiplicatively independent. Therefore the sequence $(\alpha_1^n, \ldots, \alpha_k^n)$ is Zariski-dense on the $k$-dimensional space $\mathbb{A}^k$; same holds for every infinite subsequence. In particular, $P(\alpha_1^n, \ldots, \alpha_k^n)$ cannot vanish infinitely often. Thus, it follows that $P(\alpha_1^n, \ldots, \alpha_k^n) \in \bar{\mathbb{Z}}\backslash\{0\}$ for infinitely many natural number $n$ (by hypothesis).

\bigskip
 	
  By hypothesis, since $P(X_1, 0, \ldots, 0)$ is not a constant polynomial, we let $P(X_1,0\dots,0) = \displaystyle\sum_{i=0}^d c_iX_1^i$, with $c_d \neq 0$ and therefore $(d,0,\dots,0) \in \mathcal{I}$. Let $I_1$ be an equivalence class induced by \eqref{eq2.1} corresponding to a subset, say, $\mathcal{I}_1$  of $\mathcal{I}$ such that $(d,0,\dots,0)\in \mathcal{I}_1$ (or equivalently, $\alpha_1^d\in I_1$). Since $\alpha_1, \ldots, \alpha_k$ are multiplicatively independent, we get   $\mathcal{I}_1 \subseteq \{ (c,0,\dots,0) ~|~ 0 \le c \le d  \}$. Then we have the following cases to consider:  
 \begin{enumerate}
		\item  If $(c,0, \dots,0) \in \mathcal{I}_1$ for some non-negative integer $c < d$,  then $\alpha_1^{d} \sim \alpha_1^{c}$, and therefore $\alpha_1^{d-c}$ is a root of unity. This implies that $\alpha_1$ is a root of unity. 
		\item  If $\mathcal{I}_1 = \{(d,0,\dots,0)\}$, then by Theorem \ref{TH:ALMOST-VANISHING/ALG-INT}, $\alpha_1^d$ is an algebraic integer. Therefore, $\alpha_1$ is an algebraic integer. 
	\end{enumerate}
	In both cases, we are done. 
	$\hfill\Box$
	\smallskip
	
	\noindent{\bf Proof of Theorem \ref{groupring}.}
	Let $f = \displaystyle\sum_{\sigma \in G_\alpha} \lambda_\sigma \sigma$ be  a non-zero element of $\mathbb{\overline{Q}}[G_\alpha]$. Then we have 
	$$
	f(\alpha^n) = \sum_{\sigma \in G_\alpha} \lambda_\sigma \sigma(\alpha^n) = \sum_{\sigma \in G_\alpha} \lambda_\sigma \sigma(\alpha)^n.
	$$
Therefore, by letting $\overline{X} = (X_{\sigma})_{\sigma \in G_{\alpha}}$, the linear form  $\mathcal{L}(\overline{X}) = \sum_{\sigma \in G_{\alpha}}\lambda_\sigma X_\sigma$ and $\overline{\alpha}:= (\sigma(\alpha))_{\sigma \in G_{\alpha}}$, we note that $\mathcal{L}(\overline{\alpha}^n) = f(\alpha^n)$. Hence by hypothesis, we have $\mathcal{L}(\overline{\alpha}^n)\in \overline{\mathbb{Z}} \backslash\{0\}$ for all $n \in \mathfrak{S}$. \

\smallskip
	
	The equivalence relation in \eqref{eq2.1}, induces a partition on the index set $G_\alpha$, that is, $G_\alpha = \displaystyle\cup_{j=1}^d I_j$. If possible, we suppose $\alpha$ (and therefore  all the Galois conjugates of $\alpha$) is not an algebraic integer. By Theorem \ref{TH:ALMOST-VANISHING/ALG-INT},  for every equivalence class $I_j$ we have,
	$$
	\sum_{\sigma \in I_j} \lambda_\sigma \sigma(\alpha)^n = 0 \text{ for all  but finitely many } n \in \mathfrak{S}.
	$$
	%Note that since $\alpha$ is not an algebraic integer, we have that $|I_j|>1$ for all partitions $I_j$. This is because the $\sigma(\mathcal{O}_K) = \mathcal{O}_K$. 
	Therefore, for all but finitely many values of $n\in \mathfrak{S}$, we have 
	$$
	f(\alpha^n) = \sum_{j=1}^d \sum_{\sigma \in I_j} \lambda_\sigma \sigma(\alpha)^n = 0,
	$$
	a contradiction as  $f(\alpha^n) \in \overline{\mathbb{Z}} \backslash \{0\}$ for all $n \in \mathfrak{S}$ and $\mathfrak{S}$ is an infinite set.
	 $\hfill\Box$ 
 
	\smallskip
	\vspace{.2cm}
	
 For the proof of Theorems \ref{TH:PHILI-RATH-POLY-ANALOG} and \ref{trace}, the index set, $\mathcal{I}\times G$, will be a finite subset of $\mathbb{N} \times \textrm{Gal}(K/\Q)$ for an appropriate Galois extension $K/\Q$. Then we set the linear form as 
	 \[\mathcal{L}(\overline{X}):= \mathcal{L}((X_{i,\sigma})_{(i,\sigma)\in \mathcal{I}\times G}) = \sum_{(i,\sigma) \in \mathcal{I}\times G}\sigma(\lambda_i) X_{i,\sigma}\] 
	\noindent{\bf Proof of Theorem \ref{TH:PHILI-RATH-POLY-ANALOG}}. 
	First, we can assume that $\alpha$ is not a root of unity (for otherwise, we are done). Since 
	$\Tr_{L/\Q} = \Tr_{F/\Q} \Tr_{L/F}$ for an intermediate field $F$ of $L$, we further  reduce our computations to the field $F=\Q(\alpha)$.  That is, 
	%we can see that $\Tr_{\Q(\lambda_0,\dots,\lambda_d,\alpha)/\Q}(\lambda_i \alpha^n) = \Tr_{F/\Q}(a_i\alpha^n)$ where $a_i= \Tr_{\Q(\lambda_0,\dots,\lambda_d,\alpha)/F}(\lambda_i)$. Henceforth
	 we can assume that $\lambda_i \in F$ for each $i$.  
	\smallskip
	
	 Let $K$ be the Galois closure of $F$ and its Galois group $G = \Gal{K/\Q}$. We set the index set $\mathcal{I} \times G= \{ (k,\sigma)~|~ \lambda_k \neq 0, 1\le k \le D, \sigma \in G \}$. For every algebraic number $\gamma$, we define the tuple $\overline{\gamma}:=(\sigma(\gamma)^k)_{(k,\sigma) \in \mathcal{I}\times G}$. Therefore, by hypothesis, we have $\mathcal{L}(\overline{\alpha}^n) = \Tr_{F/\Q}(P(\alpha^n))= \displaystyle\sum_{i=0}^D \Tr_{F/\Q}(\lambda_i \alpha^{in}) \in \Z$ for each $n$ in an infinite subset $\mathfrak{S} \subseteq \N$.

	If $\alpha$ is not an algebraic integer (and so are the conjugates of $\alpha$ and their powers), then by expanding the trace operator and by Theorem \ref{TH:ALMOST-VANISHING/ALG-INT},  we get for every equivalence class and the corresponding index set $I_j \times H\subset \mathcal{I}\times G$ for some subset $H\subset G$ such that  
\begin{equation}\label{eq1000}
	\sum_{i \in I_j} b_i \alpha_i^n = 0, 
	\end{equation}
	for all but finitely many $n$ where $b_i$ denotes the conjugates of $\lambda_l$ for some finite collection and $\alpha_i$ denotes some conjugate of $\alpha^l$ for some $l \le D$ appropriately. 
	
	Note that for any $(i, \sigma), (k, \tau) \in I_j \times H$, we see that $\sigma(\alpha^i) \sim \tau(\alpha^k)$ which implies that $\alpha^i \sim \delta(\alpha^k)$ for some $\delta \in G$.  Now, we claim that {\it if $\alpha^l \sim \sigma(\alpha)^k$  for some   $1 \le l,k \le D$ and for some $\sigma\in G$,  then   $l = k$}.  
In order to prove this claim, we use the properties of logarithmic Weil height \textbf{$h(x)=\log|H(x)|$}. As $l$ and $k$ are positive integers, since 
	$$l h(\alpha) = h(\alpha^l) = h(\zeta \sigma(\alpha)^k) = h(\alpha^k) = kh(\alpha),
	$$
	(where $\zeta$ is a root of unity), 
	we get  $l=k$ as $h(\alpha) > 0$, which proves the claim.
	
	  Suppose $\lambda_l\alpha^l$ is one of the terms in the above summand. Then by the claim,  \eqref{eq1000} becomes
	$$
	\sum_{\sigma \in H} \sigma(\lambda_l\alpha^{ln}) = 0
	$$
	for all but finitely many $n$. Applying the trace operator on both side,   we get $|H| \Tr_{F/\Q}(\lambda_l \alpha^{ln}) = 0$ for infinitely many $n$. Since $ |H|\geq 1$, we get the assertion. $\hfill\Box$
	\vspace{.3cm}

	\noindent{\bf Proof of Theorem \ref{trace}}.  Let $K$ be the Galois closure of $L$ and its Galois group $G =  \Gal{K/\mathbb{Q}}$.  We write the index set $\mathcal{I}\times G = \{ (i,\sigma)~|~ 1 \le i \le k, \sigma \in G \}$. For the tuple $\overline{\alpha^n}:=(\sigma(\alpha_i)^n)_{\sigma \in G, 1 \le i \le k}$, we note that $\mathcal{L}(\overline{\alpha^n}) = \mathrm{Tr}_{L/\Q}(\lambda_1\alpha_1^n+\cdots+\lambda_k\alpha^n_k) \in \Z$ for $n$ in an infinite subset $\mathfrak{S}$ of $\mathbb{N}$. 
	% \backslash \{0\}$. 
	
 Since $\alpha_1$ is not an algebraic integer, by hypothesis,   by applying Theorem \ref{TH:ALMOST-VANISHING/ALG-INT} to the linear form $\mathcal{L}(\overline{\alpha^n})$,  there exists an equivalence class and the corresponding index set $I_1$ containing $(1, \sigma_1) $ ($\sigma_1$ denotes the identity map) such that 
	\begin{equation}\label{EQN:I1-VANISHING}
		\sum_{(i,\sigma_j) \in I_1} \sigma_j(\lambda_i) \sigma_j(\alpha_i)^n = 0 = \alpha_1^n \left(\sum_{(i,\sigma_j)\in I_1} \sigma_j(\lambda_i) {\zeta_h}^{(w_{ij}-w_{11})n}\right)    
	\end{equation}
	for all but finitely many $n \in \mathfrak{S}$.  
	
	\smallskip
	
 We first prove that $I_1 = \mathcal{P} \times H$ for some subgroup $H \subseteq G$ and $\mathcal{P}  = \{i \ : \ (i, \sigma) \in I_1\} \subseteq \{1,\dots,k\}$. 
 
  Note that if $i\in\mathcal{P}$, then there exists $\sigma \in G$ such that $(i, \sigma)\in I_1$.  Therefore,  for each $i \in \mathcal{P}$, we choose $\tau_i \in G$ such that $\alpha_1 \sim \tau_i(\alpha_i)$ and hence let  $H_i   := \{ \sigma \in G~|~ \tau_i(\alpha_i) \sim \sigma(\tau_i(\alpha_i)) \}$. Then $H_i$ is a subgroup of $G$ by 
  Proposition \ref{PROP:SUBGROUP-VANISHING}.  
  
Note that $H_i = H_j$ for any $i, j \in \mathcal{P}$. For, if $\sigma\in H_i$, then $\alpha_1\sim \tau_i(\alpha_i) \sim \sigma(\tau_i(\alpha_i))$. Since $(i, \tau_i), (i, \tau_j)\in I_1$, we see that $\tau_i(\alpha_i) \sim \tau_j(\alpha_j)$. Therefore, by acting $\sigma$ on this equivalence, we get,  
$$
\tau_j(\alpha_j) \sim \tau_i(\alpha_i) \sim \sigma(\tau_i(\alpha_i)) \sim \sigma(\tau_j(\alpha_j))
$$
and hence $H_i\subseteq H_j$. Similarly, we get $H_j\subseteq H_i$.  Since $(1, \sigma_1)\in H_1$, by taking $H = H_1$, we get   $I_1 = \mathcal{P}\times H$. 

\bigskip
 
 We now   rewrite \eqref{EQN:I1-VANISHING} as 
	\[ \sum_{i \in \mathcal{P}} \sum_{ \sigma \in H} \sigma(\tau_i(\lambda_i))\sigma(\tau_i(\alpha_i^n)) = \sum_{\sigma \in H} \sigma\left(\sum_{i \in \mathcal{P}} \tau_i\left(\lambda_i\alpha_i^n\right)\right) = 0 \]
	for all but finitely many values of $n \in \mathfrak{S}$. In particular, there exists a non-negative integer $a < h$ such that there are infinitely many $n\equiv a \mod h$ with   $n\in \mathfrak{S}$.  For any such $n \equiv a \mod h$ in $\mathfrak{S}$, by combining \eqref{EQN:I1-VANISHING} and the above equality,  we get 
	\[    \sum_{\sigma \in H} \sum_{i \in \mathcal{P}} \sigma\left[\tau_i(\lambda_i)(\tau_i(\alpha_i)^a\right] = 0. \]
	Since $\Tr_{K/\mathbb{Q}}$ is invariant under the Galois action, by applying the trace operator $\Tr_{K/\Q}$, we get, 
	\[ \sum_{i\in \mathcal{P}} \Tr_{K/\Q} (\lambda_i \alpha_i^a) = 0. \]
	Since $\lambda_i, \alpha_i\in L$,  this proves the theorem. $\hfill\Box$ 
	
	\bigskip
	\bigskip
	
	\noindent{\bf Proof of Theorem \ref{rationalfunction}.}
	Let $K$ be the number field that is obtained by adjoining all the zeros and poles of $f_i(X)$'s and $\lambda_j$'s with $\mathbb{Q}$.  For each $i$, we write $f_i(X) = \displaystyle\frac{p_i(X)}{q_i(X)}$ for some coprime polynomials  $p_i(X),q_i(X) \in \mathcal{O}_K[X]$.    Let $h$ be the order of the torsion subgroup of $K^\times$.

	Since $(p_i(X), q_i(X))=1$ in $K[X]$ for all  $i$, there exist polynomials $r_i(X), s_i(X) \in \mathcal{O}_K[X]$ and $\beta_i \neq 0 \in \mathcal{O}_K$ such that 
	$$
	r_i(X) p_i(X) + s_i(X)q_i(X) = \beta_i,
	$$
	and hence $r_i(X) f_i(X) + s_i(X) = \displaystyle\frac{\beta_i}{q_i(X)}$. In order to prove $f_i(X)\in \mathcal{O}_K[X]$, first, we prove that   $q_i(X)$ is a constant polynomial in $K[X]$ for each $i$. To do this, we need to understand the conjugate polynomials of $q_i(X)$. We shall define the following.

For any number field $L$, we let  
	$$\mathfrak{V}_L := \mathcal{O}_L \backslash \{ \textrm{Zeroes of } p_i(X),q_i(X), (q_j(X)p_i(X))^h - (q_i(X)p_j(X))^h \text{ for } 1 \le i<j \le k \}.
	$$
	Note that this set is the complement of a finite set because $f_i(X)/f_j(X)$ is not constant, and also we are removing the solutions of the equation $f_i(X)/f_j(X) = \zeta_h^a$ for every $i\ne j$ and for some integer $a$.

By the definition of $\mathfrak{V}_K,$  we see that for any $\gamma \in \mathfrak{V}_K$ (or $\mathfrak{V}_L$), the tuple $(f_1(\gamma), \ldots, f_k(\gamma))$ is a non-degenerate tuple.   Therefore, by Theorem \ref{TH:ALMOST-VANISHING/ALG-INT}, $f_i(\gamma) \in \mathcal{O}_K$ (or $\mathcal{O}_L$) for all $i$.  Hence,  the value $\beta_i/q_i(\gamma) \in \mathcal{O}_K$ for each $\gamma \in \mathfrak{V}_K$. Note also that $\mathfrak{V}_K$ contains all but finitely many integers in it.

	If possible, for some $i$, we suppose $q_i(X)$ is a non-constant polynomial. Let $L$ be the Galois closure of $\Q(\beta_i,\text{coefficients of }q_i(X))$. Then the polynomial $Q_i(X) = \displaystyle\prod_{\sigma \in \textrm{Gal}(L(X)/\Q(X)) }\sigma(q_i(X))$ is also a non-constant polynomial in $\mathbb{Q}[X]$. Consider $\mathfrak{V}_L$ and conclude that $\beta_i/q_i(\gamma) \in \mathcal{O}_L$ for each $\gamma\in \mathfrak{V}_L$. Therefore,   $N_{L/\Q}(\beta_i/q_i(\gamma)) \in \mathbb{Z}$ for every $\gamma \in \mathfrak{V}_L$.   Note that for all but finitely many $n \in \mathbb{Z}$ lie in $\mathfrak{V}_L$ and hence we have $N_{L/\Q}(\beta_i/q_i(n)) = N_{L/\Q}(\beta_i)/Q_i(n) \in \Z$. Since $|Q_i(n)| \to \infty$ as $|n| \to \infty$, we obtain $|N_{L/\Q}(\beta_i)/Q_i(n)| \to 0$ as $n \to \infty$.  However,  this is a sequence of integers and therefore we conclude that $N_{L/\Q}(\beta_i) = 0$.  This implies $\beta_i = 0$, which is a contradiction. Hence $q_i(X)$ must be constant and hence $f_i(X) \in K[X]$ for each $i$.

	We now proceed to show that $f_i(X) \in \mathcal{O}_K[X]$. Let  $f_i(X) =\displaystyle \frac{p_i(X)}{\gamma_i}= \frac{1}{\gamma_i} \sum_{j=0}^d b_jX^j$ for some $\gamma_i, b_j \in \mathcal{O}_K$. Suppose there exists a prime ideal $\mathfrak{P}$ in $\mathcal{O}_K$ such that $v_\mathfrak{P}(b_j/\gamma_i)<0$ for some $j$.  We choose a number field $L$ containing $K$ having a prime ideal $\mathfrak{Q}$ in $\mathcal{O}_L$ lying above $\mathfrak{P}$ such that the ramification index (say $e$) is greater than $2d$.  Note that such $L$ can be chosen by adjoining the appropriate root of unity to $K$ to get the desired.  Therefore, for any $\delta \in K^\times$, we have $v_\mathfrak{Q}(\delta) = e v_\mathfrak{P}(\delta)$.  By the choice of $e$, we conclude that 
	\begin{equation}\label{eq4.1}
		e \min_{0\leq r\leq d}{v_\mathfrak{P}(b_r/\gamma_i)} + d < 0 .
	\end{equation}
Now since $\mathfrak{V}_L$ is the complement of a finite set, we choose $\alpha \in \mathfrak{V}_L$ such that $v_{\mathfrak{Q}}(\alpha)=1$. Then 
$
\displaystyle v_\mathfrak{Q}(f_i(\alpha)) = \min_{r} v_{\mathfrak{Q}}\left(\frac{b_r}{\gamma_i} +r\right) 
$  because for any two distinct non-negative numbers $r, s \le d$, due to \eqref{eq4.1}, we have, 
 $
	\displaystyle v_\mathfrak{Q}\left(\frac{b_s\alpha^s}{\gamma_i}\right) \neq v_\mathfrak{Q}\left(\frac{b_r\alpha^r}{\gamma_i}\right), $ as $e > 2d$ (If they are equal, we take absolute values and consider their difference to get a contradiction). 
	Therefore,
	$$
	v_\mathfrak{Q}(f_i(\alpha))   \le \min_r e\left[v_\mathfrak{P}\left(\frac{b_r}{\gamma_i}\right)\right]+d< 0,
	$$
  a contradiction to the fact that  $f_i(\alpha) \in \mathcal{O}_L$ as $\alpha \in \mathfrak{V}_L$. $\hfill\Box$
	\begin{comment}
	\begin{Rem}
	For the second part, to conclude that $f_i(x) \in \mathcal{O}_K[x]$ given that $f_i(x)\in K[x]$ it is necessary to take a ramified extension. This is because we have integer-valued polynomials that are not in $\mathcal{O}_K[x]$ (For example, if we take $f(x)= x(x-1)/2$, then $f(\Z) \subseteq \Z$). 
	\end{Rem}
	\end{comment}
	\
	\section{Proofs of Theorems \ref{thm1.1} and \ref{finite}}
	
	\noindent{\bf Proof of Theorem \ref{thm1.1}}. 
	By Lemma \ref{lem0}, it is enough to prove the assertion for $\mathbb{Q}_p$ for every prime number $p$. 
	
	Let $p$ be a given prime number. Assume that  $\alpha_1$ is not integral over $\mathbb{Z}_p$ and $K$ be the Galois closure of $\mathbb{Q}_p(\alpha_1)$. All the other Galois conjugates of $\alpha_1$ are $\alpha_2, \ldots, \alpha_k$ for some integer $k$. It is  also enough to prove the case when $b_3=b_4 = \cdots = b_k =0$ and  the proof for the general case follows verbatim. 
	
	Assume that $\alpha_1$ is not integral over $\mathbb{Z}_p$ (and so is  $\alpha_2$).  For simplicity, we write $\alpha_1 = \alpha$ and $\alpha_2 = \beta$.  Then $\alpha^{-1}$ and $\beta^{-1}$ are integral over $\mathbb{Z}_p$ and let  the  characteristic polynomial $f_{K|\mathbb{Q}_p}(x):= f(x)$ of $\alpha^{-1}$   satisfies the assertion in Lemma \ref{lem1}.  Since $\beta$ is a Galois conjugate of $\alpha$, we see that the characteristic polynomial of $\beta^{-1}$ is $f(x)$ itself. Write the unique maximal ideal $p\mathbb{Z}_p$ of $\mathbb{Z}_p$ by $\mathfrak{P}$. 
	\bigskip
	
	If $f(x) = x^d+a_{d-1}x^{d-1}+\cdots+a_0$, then  $f(\alpha^{-1}) = 0$  and  $f(\beta^{-1}) = 0$. Hence we get
	\begin{equation}\label{eq5.2}
		a_0 +a_1\alpha^{-1}+ \cdots + a_{d-1}\alpha^{-d+1}+ \alpha^{-d}= 0 = a_0+ a_1\beta^{-1}+\cdots + a_{d-1}\beta^{-d} 
	\end{equation}
	Then, for any integer $\ell \geq 0$, multiplying by $ \alpha^{d+\ell}$ both sides of \eqref{eq5.2}, we get 
	\begin{equation}\label{eq5.3}
		\alpha^{\ell} = - a_{d-1}\alpha^{\ell+1}-\cdots  - a_0\alpha^{d+\ell} \mbox{ with } a_i\in \mathfrak{P}
	\end{equation}
	and
	\begin{equation}\label{eq5.4}
		\beta^{\ell} = - a_{d-1}\beta^{\ell+1}-\cdots  - a_0\beta^{d+\ell} \mbox{ with } a_i\in \mathfrak{P}
	\end{equation}
	by Lemma \ref{lem1}. Now, for any integer $\ell\geq 0$, let $M_\ell$ be a $\mathbb{Z}_p$-submodule of $K$ spanned by $b_1\alpha^{\ell+1}+b_2\beta^{\ell+1}, \ldots,$ $ b_1\alpha^{\ell+d}+b_2 \beta^{\ell+d}$. By Lemma \ref{lem2}, it is clear that $M_\ell$ is  a non-zero $\mathbb{Z}_p$-submodule of $K$.   Hence, by \eqref{eq5.3} and \eqref{eq5.4},    we have 
	\begin{equation}\label{eq5.5}
		b_1 \alpha^\ell +b_2\beta^\ell \in \mathfrak{P}M_\ell \mbox{ for any integer } \ell\geq 0.
	\end{equation}
	Note that for any nonnegative integers $\ell_1$ and $\ell_2$, we have
	\begin{equation}\label{eq5.6}
		M_{\ell_1} \subset M_{\ell_2} \mbox{ whenever } \ell_1 < \ell_2.
	\end{equation}
	It is enough to prove that for $\ell_1 = \ell$, $\ell_2 = \ell+1$, we have $b_1\alpha^{\ell+1}+ b_2\beta^{\ell+1} \in M_{\ell_2}$  (and then inductively we can get the general assertion). Since $f(\alpha^{-1}) = 0 = f(\beta^{-1})$, multiplying by $\alpha^{\ell+1+d}$ on both sides, similarly, by  $\beta^{\ell+1+d}$, we get $b_1\alpha^{\ell+1}+b_2 \beta^{\ell+1} \in M_{\ell_2}$, as desired.
	
	Now we claim that for any integer $\ell\geq 0$ and any integer $m\geq 0$, we have 
	\begin{equation}\label{eq5.7}
		b_1 \alpha^\ell+b_2 \beta^\ell \in \mathfrak{P}^{m+1}M_{\ell+dm}.
	\end{equation}  
	Let $\ell$ be any nonnegative integer and $m=0$. Then \eqref{eq5.7} is true by \eqref{eq5.5}. Hence, we shall assume that \eqref{eq5.7} holds true for $\ell$ and for some integer $m\geq 1$. That is, we have $b_1 \alpha^\ell +b_2\beta^\ell\in\mathfrak{P}^{m+1}M_{\ell+dm}$ and  we prove \eqref{eq5.7} holds true for $\ell$ and $m+1$. For any integer $i$ with $\ell+dm+1\leq i\leq \ell+d(m+1)$,  we have $b_1\alpha^i + b_2\beta^i \in M_{\ell+dm}$. By \eqref{eq5.5} and \eqref{eq5.6}, we get 
	$$
	b_1 \alpha^i+b_2 \beta^i\in\mathfrak{P}M_i \subset \mathfrak{P}M_{\ell+d(m+1)} \mbox{ for all integers } i \mbox{ with }  \ell+dm+1\leq i\leq \ell+d(m+1)
	$$
	and hence we get
	\begin{equation}\label{eq5.8}
		M_{\ell+dm} \subset \mathfrak{P}M_{\ell+d(m+1)}.
	\end{equation}
	Since, by the induction hypothesis, we have   $ b_1\alpha^\ell+ b_2\beta^\ell \in \mathfrak{P}^{m+1}M_{\ell+dm}$,  by \eqref{eq5.8}, we arrive at 
	$$
	b_1\alpha^\ell+b_2\beta^\ell \in \mathfrak{P}^{m+2}M_{\ell+d(m+1)}
	$$
	as desired. 
	
	\bigskip
	
	Now to finish the proof, we take $\ell = 0$ and $m = \mathit{v}_p((b_1+b_2)d)$.  By hypothesis, we know that ${\mathrm{Tr}}_{K|\mathbb{Q}_p}(M_{a-d}) \subset \mathbb{Z}_p$ for all integers $a\leq [d\log_2((b_1+b_2)d)]+1$.  Since $dm < d\log_2((b_1+b_2)d)+1$, we get ${\mathrm{Tr}}_{K|\mathbb{Q}_p}(M_{dm}) \subset \mathbb{Z}_p$. Therefore, since  $b_1+b_2 = b_1\alpha^0+b_2\beta^0 \in \mathfrak{P}^{m+1}M_{dm}$, we see that  
	$(b_1+b_2)d = {\mathrm{Tr}}_{K|\mathbb{Q}_p}(b_1\alpha^0+b_2\beta^0) \in {\mathrm{Tr}}_{K|\mathbb{Q}_p}(\mathfrak{P}^{m+1}M_{dm}) \subset \mathbb{Z}_p$.  Therefore,  we get, $(b_1+b_2)d  \in \mathfrak{P}^{m+1}$ which implies that  the power of $p$ dividing $(b_1+b_2)d$ is at least $m+1$, a contradiction. Hence the theorem. $\hfill\Box$

	%\vspace{.3cm}\textbf{
	%	Before starting the proof of Theorem \ref{finite}, we recall that $m_n:=\lambda_1\alpha_1^n + \lambda_2\alpha_2^n$.} \\
	\noindent{\bf Proof of Theorem \ref{finite}}. Given that 
	 $\displaystyle C  = 2+ \left\lceil\frac{1}{2}\max_\mathfrak{P}{v_\mathfrak{P}(\lambda_1\lambda_2(\alpha_1-\alpha_2)^2)}\right\rceil + \max_\mathfrak{P}(|v_\mathfrak{P}{(\lambda_1\lambda_2)}|)
$ where $\mathfrak{P}$ runs through all the prime ideals in $\mathcal{O}_K$ and  
$\lambda_1\alpha_1^n+\lambda_2\alpha_2^n \in \mathcal{O}_K$ for all $1 \le n \le C$. 

If possible, we assume that $\alpha_1$ is not an algebraic integer.   Then  there exists a prime ideal $\mathfrak{P}$ of $\mathcal{O}_K$ such that $v_{\mathfrak{P}}(\alpha_1) < 0$. 

We first claim that $v_{\mathfrak{P}}(\alpha_1)=v_{\mathfrak{P}}(\alpha_2)$. To show this, for each $i> |v_{\mathfrak{P}}(\lambda_1)|$, we have that $v_\mathfrak{P}(\lambda_1 \alpha_1^i) = v_\mathfrak{P}(\lambda_1) + i v_\mathfrak{P}(\alpha_1) \le v_\mathfrak{P}(\lambda_1) -i < 0$. Using the fact that $v_\mathfrak{P}(x+y)=\min\{v_\mathfrak{P}(x), v_\mathfrak{P}(y)\}$, when $v_\mathfrak{P}(x)\neq  v_\mathfrak{P}(y)$, and that  
	$ v_\mathfrak{P}(\lambda_1 \alpha_1^i + \lambda_2 \alpha_2^i) \ge 0 $ as  $1 \le i \le C$, we conclude that   
	$$
	v_\mathfrak{P}(\lambda_2\alpha^i_2)=v_\mathfrak{P}(\lambda_1\alpha^i_1)\implies i|v_\mathfrak{P}(\alpha_1)-v_\mathfrak{P}(\alpha_2)|\leq |v_\mathfrak{P}(\lambda_1\lambda_2)|
	$$
	holds for each $i > |v_\mathfrak{P}(\lambda_1)|$. Now, by choosing $i > |v_\mathfrak{P}(\lambda_1\lambda_2)|$, we conclude that $v_\mathfrak{P}(\alpha_1) = v_\mathfrak{P}(\alpha_2)$ as $v_\mathfrak{P}$ takes integer values. In particular, $v_\mathfrak{P}(\alpha_2)<0$.
	
We note that for each $n\geq 1$, 	 
		\begin{equation}\label{eq5.1}
			V_n:= [\lambda_1\alpha_1^n+\lambda_2\alpha_2^n] [\lambda_1\alpha_1^{n+2}+\lambda_2\alpha_2^{n+2}] -  [\lambda_1\alpha_1^{n+1}+\lambda_2\alpha_2^{n+1}]^2 = \lambda_1\lambda_2 \alpha_1^n\alpha_2^n [\alpha_1 - \alpha_2]^2 = V_1 [\alpha_1\alpha_2]^{n-1}.  
		\end{equation}
		For $ n > \max\left\{(v_\mathfrak{P}(V_1))/2+1,|v_\mathfrak{P}(\lambda_1\lambda_2)|\right\}$, we have 
		\[v_\mathfrak{P}(V_n)= v_\mathfrak{P} (V_1) + (n-1)v_\mathfrak{P} (\alpha_1\alpha_2) \le v_\mathfrak{P}(V_1) - 2(n-1) < 0. \]
In particular, $v_\mathfrak{P}(V_C) < 0$, as $C  > \max\left\{(v_\mathfrak{P}(V_1))/2+1,|v_\mathfrak{P}(\lambda_1\lambda_2)|\right\}$. 	This is not possible, because	
		  one notes that $V_n   \in \mathcal{O}_K$ for all $n\leq C$ and  in particular, $V_C \in \mathcal{O}_K$.  Therefore  $\alpha_1\in \mathcal{O}_K$. Similarly, we can prove $ \alpha_2\in \mathcal{O}_K$. $\hfill\Box$
 
	\smallskip
	
	\begin{Rem}\label{REM:TH-finite-difficulties}
		One may try to generalise this argument for $k \ge 3$. There are issues with the valuation argument for more than 2 variables. We may use Hankel determinants of matrices with entries consisting only of $\lambda_1\alpha_1^i+\cdots+\lambda_k\alpha_k^i$ to arrive at an equation very similar to \eqref{eq5.1}. Proceeding in the same manner from there, one may obtain that $\displaystyle\prod_{i=1}^k \alpha_i \in \mathcal{O}_K$. However, it is not possible to obtain a bound purely depending on $\lambda_1\alpha_1^i+\cdots+\lambda_k\alpha_k^i,\lambda_i$, by induction for the following reason: When we try to do induction over $k$, we know the values $\lambda_1\alpha_1^i+\cdots+\lambda_k\alpha_k^i$ only for $k$ terms and do not know for $k-1$ terms. The process will follow through but we won't be able to determine the constant $C$, if we assume that $|\alpha_i|_\mathfrak{P} \le 1$ for some $i$ by this method.  
	\end{Rem}
	\vspace{.3cm}

	\noindent{\bf Acknowledgments.} We are thankful to the referees for carefully going through the earlier draft and suggesting many useful changes to make the article readable. This work is part of the SERB-MATRICS project and the last author is thankful to the SERB, India.


\begin{thebibliography}{9999}
		\bibitem{bomb}
		E. Bombieri and W. Gubler,  {\it Heights in Diophantine geometry}, New Mathematical Monographs, vol. 4, {\it Cambridge University Press}, Cambridge, 2006.
		\bibitem{corv} 
		P. Corvaja and U. Zannier,   On the rational approximation to the powers of an algebraic number: Solution of two problems of Mahler and Mendes France, {\it Acta Math.}, {\bf 193} (2004), 175--191.
		\bibitem{FATOU}
		{P. Fatou}, Sur les séries entières à coefficients entiers. \textit{C. R. Acad. Sci.}, Paris 138, (1904) 342--344.
		\bibitem{smit}
		B. de Smit, Algebraic numbers with integral power traces, {\it J. Number Theory} {\bf 45} (1) (1993)  112--116.
		\bibitem{Hindry}
		M. Hindry and J. H. Silverman, Diophantine Geometry: An Introduction, Graduate Texts in Mathematics 201,{\it Springer-Verlag}, New York, 2000.
		\bibitem{kul}
		A. Kulkarni, N. Mavraki and K. D. Nguyen,  Algebraic approximations to linear combinations of powers: An extension of results by Mahler and Corvaja-Zannier, {\it Trans. Amer. Math. Soc.}, {\bf 371} (6) (2019) 3787--3804.
		
		\bibitem{lang}
		S. Lang, Fundamentals of Diophantine Geometry, {\it Springer-Verlag}, 1983.
		\bibitem{MACDONALD}
		I. G. Macdonald, Symmetric functions and Hall polynomials. With contributions by A. V. Zelevinsky. \textit{Oxford University Press} (1998)
		\bibitem{rath}
		P. Philippon and P. Rath,  A note on trace of powers of algebraic numbers, {\it J. Number Theory}, {\bf 219} (2021), 198--211.
		\bibitem{PISOT}
		{C. Pisot}, La répartition modulo 1 et les nombres algébriques. \textit{Ann. Sc. Norm. Super. Pisa, II. Ser.} 7, 205--248 (1938)
		\bibitem{polya}
		G. Polya,   \"{U}ber ganzwertige ganze Funktionen, {\it Rend. Circ Mat. Palermo}, {\bf 40} (1915), 1--16.
		
		\bibitem{schmidt}
		W. M. Schmidt, {\it Diophantine  Approximations and Diophantine Equations}, Lecture Notes in Math., {\bf 1467}, {\it Springer}, Berlin, 1991.
		
		\bibitem{zannier}
		U. Zannier, {\it Some Applications of Diophantine Approximation to Diophantine Equations} (with special emphasis on the Schmidt Subspace Theorem), {\it Forum}, Udine, 2003.
		
	\end{thebibliography}
\end{document}